\documentclass[a4paper]{article}
\usepackage{onecolpceurws}
\institution{}

\usepackage[english]{babel}
\usepackage[latin1]{inputenc}
\usepackage{cmbright}
\usepackage[T1]{fontenc}
\usepackage{lmodern}
\usepackage{cite}
\usepackage{xcolor}
\usepackage{hyperref}
\definecolor{darkgreen}{rgb}{0,0.4,0}
\definecolor{BrickRed}{rgb}{0.65,0.08,0}
\hypersetup{colorlinks=true,linkcolor=blue,citecolor=darkgreen,filecolor=BrickRed,urlcolor=darkgreen}

\usepackage{graphicx} 
\usepackage{amsmath,amssymb,amsthm}

\newcommand{\walksym}{\omega}
\newcommand{\walk}[1]{\walksym_{#1}}

\newcommand{\stepset}{\mathcal{S}}
\newcommand{\Hc}{\mathcal{H}} 
\newcommand{\C}{\mathbb{C}} 
\newcommand{\N}{\mathbb{N}}
\newcommand{\Z}{{\mathbb Z}}
\newcommand{\R}{{\mathbb R}}
\newcommand{\Rc}{\mathcal{R}}

\def\upos{\{u^{> 0}\}}

\def\uneg{\{u^{< 0}\}}

\newtheorem{theo}{Theorem}[section]

\newtheorem{definition}[theo]{Definition}
\providecommand{\keywords}[1]{\textbf{\textit{Keywords: }} #1}

\begin{document}
\author{Cyril Banderier\\ CNRS \& Univ. Paris Nord, France\\\url{http://lipn.fr/~banderier/}\\
\includegraphics[width=3mm]{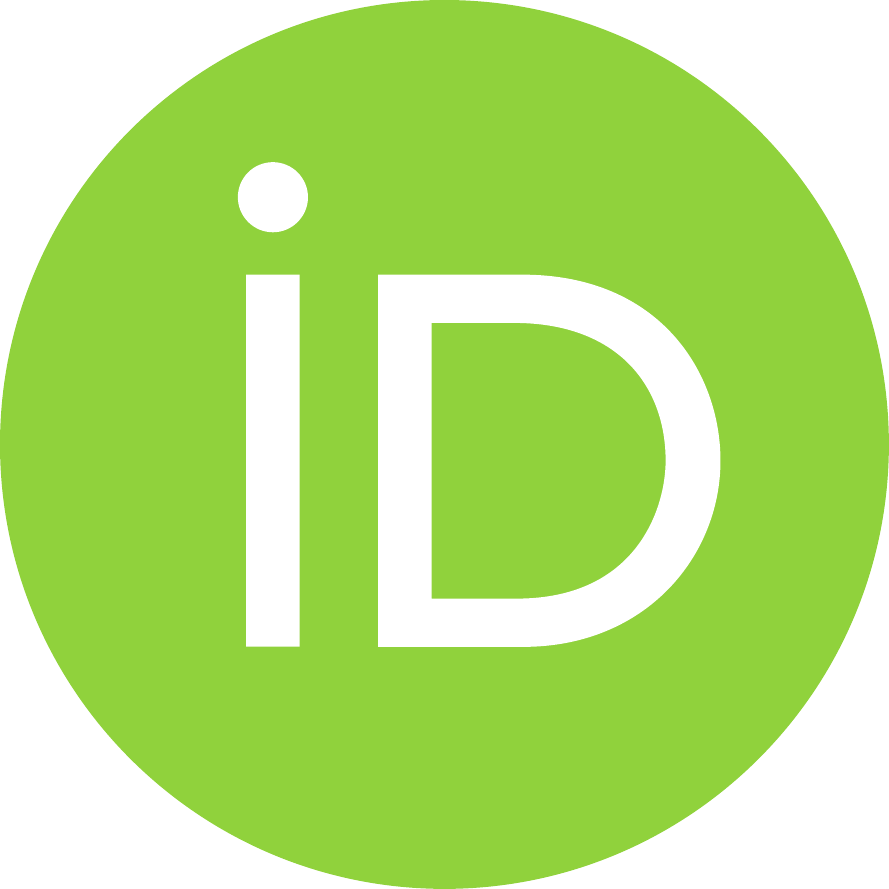} {\small\url{https://orcid.org/0000-0003-0755-3022}}  
\and 
Michael Wallner\\LaBRI, Universit\'e de Bordeaux, France\\\url{http://dmg.tuwien.ac.at/mwallner/} \\
\includegraphics[width=3mm]{orcid} \small{\url{https://orcid.org/0000-0001-8581-449X}}}

\title{Local time for lattice paths and the associated limit laws} 
\date{}
\maketitle

\begin{abstract}
For generalized Dyck paths (i.e., directed lattice paths with any finite set of jumps), we analyse their local time at zero (i.e., the number of times the path is touching or crossing the abscissa).
As we are in a discrete setting, the event we analyse here is ``invisible'' to the tools of Brownian motion theory.
 
It is interesting that the key tool for analysing directed lattice paths, which is the kernel method, is not directly applicable here. Therefore, we introduce a variant of this kernel method to get the trivariate generating function (length, final altitude, local time): this leads to an expression involving symmetric and algebraic functions.

We apply this analysis to different types of constrained lattice paths (meanders, excursions, bridges, \ldots).
Then, we illustrate this approach on ``basketball walks'' which are walks defined by the jumps $-2,-1,0,+1,+2$. 
We use singularity analysis to prove that the limit laws 
for the local time are (depending on the drift and the type of walk)
the geometric distribution, the negative binomial distribution, the Rayleigh distribution, or the half-normal distribution
(a universal distribution up to now rarely encountered in analytic combinatorics).
\end{abstract}
\vskip 32pt


\section{Introduction}
\label{sec:intro}

This article continues our series of investigations on the enumeration, generation, and asymptotics of directed lattice paths~\cite{AsinowskiBacherBanderierGittenberger18,AsinowskiBacherBanderierGittenberger18a,BaFl02,BaWa14,BaWa16,BanderierGittenberger06,KKKK16,Wallner16b}.
Such lattice paths are a fundamental combinatorial structure ubiquitous in computer science (evolution of a stack, bijections with trees, permutations, \ldots), probability theory (linked with random walks or queuing theory), and
statistical mechanics (as basic building blocks for more general 2D models), to name a few.

\pagebreak

\setlength{\abovecaptionskip}{2pt}
\def\w{70mm}
 \begin{table}[t]
 \small
 \begin{center}\renewcommand{\tabcolsep}{3pt}
 \begin{tabular}{|c|c|c|}
 \hline
 & ending anywhere & ending at 0\\
 \hline
 \begin{tabular}{c} unconstrained \\ (on~$\Z$) \end{tabular}
 & \begin{tabular}{c}  \\[-2mm]{\includegraphics[width=\w]{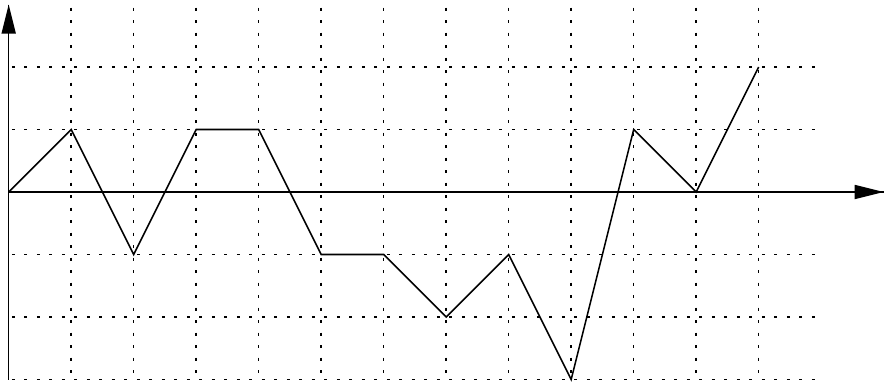}} \\  walk/path ($\cal W$)  \end{tabular}
 & \begin{tabular}{c}  \\[-2mm]{\includegraphics[width=\w]{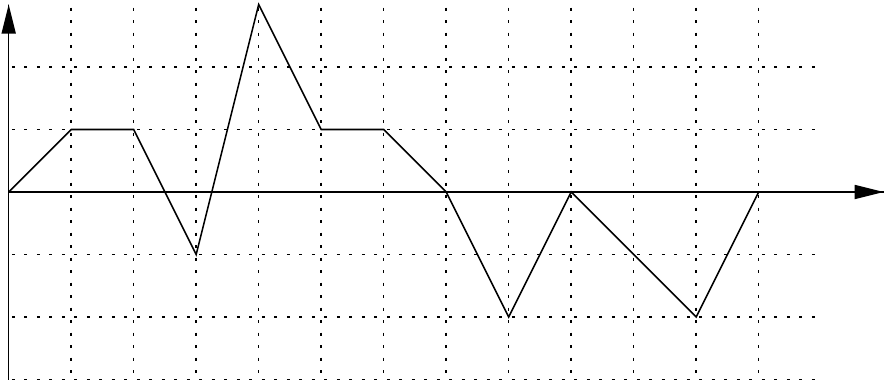}} \\ bridge ($\cal B$) \end{tabular} \\
 \hline
 \begin{tabular}{c}constrained\\ (on $\N$) \end{tabular}
 & \begin{tabular}{c}  \\[-2mm]{\includegraphics[width=\w]{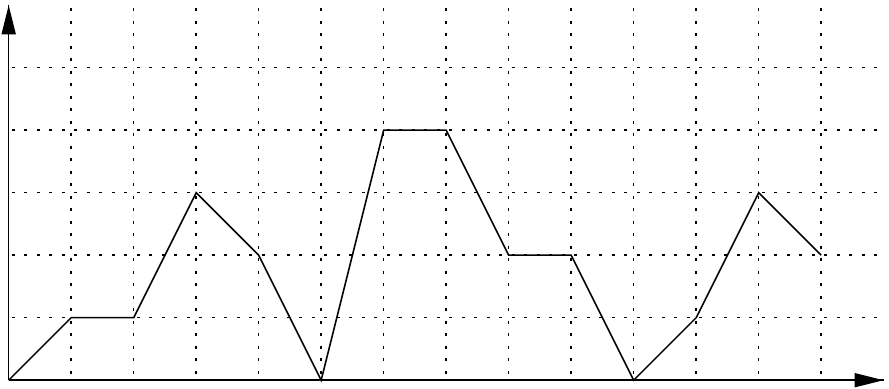}} \\  meander ($\cal M$)\\  \end{tabular}
 & \begin{tabular}{c}  \\[-2mm]{\includegraphics[width=\w]{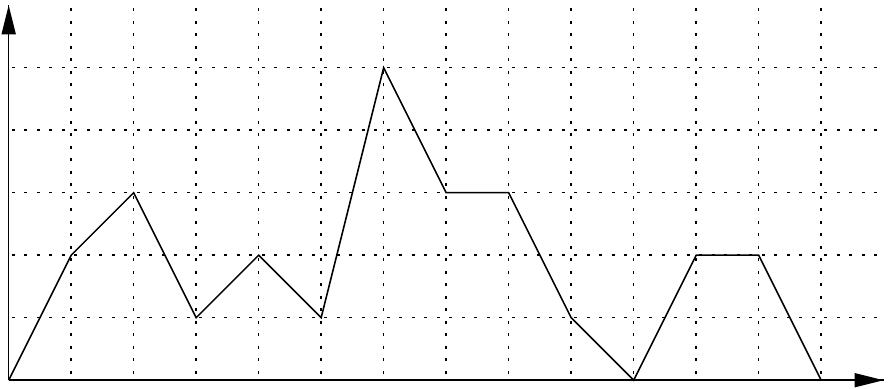}} \\  excursion ($\cal E$)\\  \end{tabular}\\
 \hline
 \end{tabular}
 \end{center}
 \caption{Four types of paths: walks, bridges, meanders, and excursions.}\label{fig-4types} 
 \end{table}

Let us give a definition of the lattice paths we consider:
 \begin{definition}[Jumps and lattice paths] \label{def:LP}
 A \emph{step set} $\stepset \subset \Z^2$ is a finite set of vectors $\{ (x_1,y_1), \ldots, (x_m,y_m)\}$. 
An $n$-step \emph{lattice path} or \emph{walk} is a sequence of vectors $(v_1,\ldots,v_n)$, such that $v_j$ is in~$\stepset$. 
Geometrically, it may be interpreted as a sequence of points $\walksym =(\walk{0},\walk{1},\ldots,\walk{n})$ where $\walk{i} \in \Z^2,~\walk{0} = (0,0)$ (or another starting point)
and $\walk{i}-\walk{i-1} = v_i$ for $i=1,\ldots,n$.
The elements of $\stepset$ are called \emph{steps} or \emph{jumps}. 
The \emph{length} $|\walksym|$ of a lattice path is its number~$n$ of jumps. 
 \end{definition}
We restrict our attention to \emph{directed paths} which are defined by the fact that, for each jump $(x,y) \in \stepset$, one must have $x \geq 0$. 
Lattice paths can have different additional constraints shown in Table~\ref{fig-4types}.

Note that it is possible to encode lattice paths in words. Then, the constrained lattice paths 
we consider can be enumerated by context-free grammars~\cite{LabelleYeh90,MerliniRogersSprugnoliVerri99,Duchon00}.
One drawback of the grammar approach is that it is not easy to get universal asymptotic results from it,  
even if it is possible to establish generic results on the associated critical exponents~\cite{BaDr13}. 
Also, the grammar approach quickly leads to heavy case-by-case computations (resultants of equations of huge degree)
as soon as the set of jumps $\stepset$ contains a large jump.

In this article, we show how to proceed for the enumeration and the asymptotics in these harder cases: 
our techniques are relying on the ``kernel method'' 
which (contrary to the context-free grammar approach) offers access to the {\em generic} structure of the final generating functions 
and the {\em universality} of their asymptotics via singularity analysis~\cite{BaFl02,flaj09}. 
 
The following convenient notation, a variant of the Omega operator of MacMahon, will be another of our ingredients:
\begin{equation*}\upos := \sum_{n>0} u^n [u^n] \text{\qquad and \qquad} \uneg:= \sum_{n< 0} u^n [u^n]
\end{equation*}
where $[u^n] g(u)$ stands for the coefficient of $u^n$ in a (Laurent) power series $g(u)$ from $\C((u))$ or $\C[u,1/u][[z]]$.
MacMahon introduced this operator on rational functions, in order to get binomial identities or integer partition formulae~\cite[Section VIII]{MacMahon15}.
In the late 1990s, this operator experienced a strong revival, mostly by work/packages of Andrews, Paule, Riese, and Han~\cite{Andrews01, Han03}.
Another nice application of the Omega operator is the proof of D-finiteness of some walks in the quarter plane 
by Bousquet-M\'elou and Mishna~\cite{BousquetMelouMishna10}.
In this article we make use of this operator, not on sums or products of rational functions (like it is the case for integer partitions, 
or for quarter plane walks),
but on the level of functional equations involving algebraic functions.

\pagebreak

\section{Generating function for the local time at $0$}

The number of times the lattice path is exactly at altitude $0$ is an easy parameter to catch via combinatorial decompositions
(see analysis of the number of returns to zero for excursions, via a decomposition into arches, by Banderier and Flajolet~\cite{BaFl02}).
In order to get the {\em local time at $0$} (as defined in Figure~\ref{axisCrossings}), it remains to capture a more subtle parameter: 
the number of steps which are crossing the $x$-axis (without actually starting or ending at altitude~$0$).
For any family of lattice paths, let $P(u)=1/u^c+\dots+u^d$ be the Laurent polynomial encoding the jumps allowed at each step.
For example, the ``basketball walks'' which we considered in~\cite{KKKK16} are walks with jumps in the set $\{-2, -1, 0,+1,+2\}$, 
and we get $P(u) := 1/u^{2} + 1/u + 1+ u+u^2$. 
More generally, each jump $i$ may get a weight~$p_i$, which gives \[P(u)=\sum_{i=-c}^d p_i u^i.\]

Feller~\cite{Feller68,Feller71}, Cz\'aki and Vincze~\cite{CzakiVincze61} and Jain~\cite{Jain66}, 
considered $P(u)=u/2+(1/2)u$, Wallner~\cite{Wallner16b} considered $P(u)= p_{-1}/u+p_0+ p_1 u$,
and we show here which new method is needed to tackle more general $P(u)$.

\begin{figure}[bh]
	\centering
\includegraphics[width=172mm]{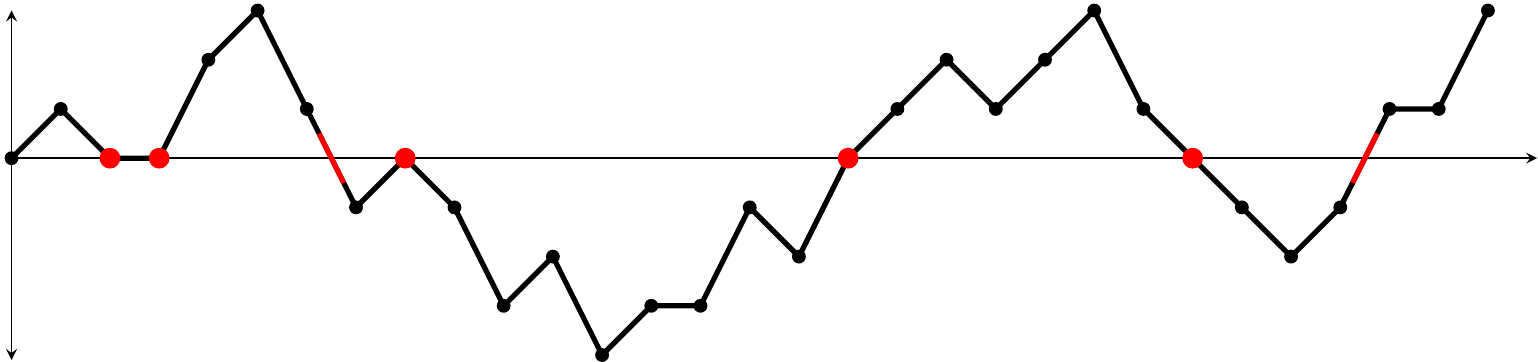}
\caption{\label{axisCrossings} 
We consider the number of times a lattice path is touching or crossing the $x$-axis.
The \emph{local time} at altitude~$h$ of a path $((x_0,y_0),\dots,(x_n,y_n))$ is $\# \{ 0< t< n$ such that ($y_t > h$ and $y_{t+1} < h$) or ($y_t < h$ and $y_{t+1} > h$) or ($y_t = h) \}$.
The same set without the last condition for $h=0$ defines $x$-axis crossings.  
In the figure above, the local time at zero is $7$, the number of returns to zero is $5$ (large red vertices), and the number of $x$-axis crossings is $2$ (red line segments).
In this article, we show how to capture the local time at~$0$;
 it is then straightforward to do it at any other altitude. }
\end{figure}

Figure~\ref{axisCrossings} illustrates that
\begin{center} local time at zero \quad = \quad number of returns to zero \quad + \quad number of $x$-axis crossings. \end{center}
As the generating function of the returns to zero, and the corresponding limit laws
are known (see~\cite{BaFl02,Wallner16b}), we can now focus on the number of $x$-axis crossings.

\begin{theo}[Generating function for number of $x$-axis crossings]\label{theo:closedform}
Let $w_{n,k,j}$ be the number of walks of length $n$ which end at altitude $k$, 
and have $j$ crossings of the $x$-axis.
Then the generating function 
\begin{equation*}W(z,u,q) = \sum_{n,j \geq 0, k\in \Z} w_{n,k,j} z^n u^k q^j
= \sum_{k \in \Z} W_k(z,q) u^k
\end{equation*}
is algebraic and expressible in terms of the roots $u_i(z)$ of $1-zP(u)=0$.
\end{theo}
\begin{proof}[Proof (Sketch)]
A step by step decomposition of the walk gives the following functional equation 
(where we write $W_k$ for $W_k(z,q)$ for readability, and where $q$ encodes the $x$-axis crossings):

\begin{align*}
	W(z,u,q) = 1 +z P(u) W(z,u,q)
&- z \left( \sum_{k=-d+1}^{-1} \{u^{>0}\} (P(u) u^k) W_{k} + \sum_{k=1}^{c-1} \{u^{<0}\} (P(u) u^k) W_k \right)\\
&+zq \left( \sum_{k=-d+1}^{-1} \{u^{>0}\} (P(u) u^k) W_{k} + \sum_{k=1}^{c-1} \{u^{<0}\} (P(u) u^k) W_k \right).
\end{align*}
This equation can be read as ``$W=$ empty walk {\em or} the walk is at some altitude encoded by $u^k$ 
and multiply by $P(u)$ to do all the jumps, then we remove the ones crossing the $x$-axis, and re-add them with a marker $q$.''
This is conveniently rewritten as
\begin{align*}
	\label{eq:funceq}
	(1 - z P(u)) \left( \sum_{k\in \Z} W_k u^k \right) = 1 - z (1-q) \left( \sum_{k=-d+1}^{-1} \{u^{>0}\} (P(u) u^k) W_{k} + \sum_{k=1}^{c-1} \{u^{<0}\} (P(u) u^k) W_{k} \right).
\end{align*}

Here, one may think that we could apply the classical kernel method (see e.g.~\cite{BaFl02}): 
one substitutes $u$ by any root $u(z)$ of $1-zP(u)$,
as this will cancel the left-hand side, and thus doing that for all roots 
will lead to a system of $c+d$ unknowns, with $c+d$ independent relations. 
Hence, bingo, solving this system gives a closed-form formula for $W(z,u,q)$!?
Unfortunately, this is not working: indeed, in such an equation, if one substitutes a variable by another expression,
one has to take care to stay in the ring of formal power series
(in order to avoid non-trivial zero divisors, as exemplified by the phenomenon $(1-u) \sum_{k\in\Z} u^k =0$). 
In our equation, $F$ belongs to $\C[u,1/u][q][[z]]$. As the exponents of $u$ range from $-\infty$ to $+\infty$,
it is not legitimate to substitute $u$ by our Puiseux series $u(z)$, because this would lead to arbitrary negative and positive powers of $z$.
We need to adapt the kernel method and to transform it into what we call a ``bilateral kernel method'':
\begin{enumerate}
\item Extract the positive part $\upos$. This gives the following equation
\begin{align*}	
\upos \left( (1-zP(u))\sum_{k=-d+1}^{\infty} W_k u^k \right)	 = - 
	 	z(1-q)\sum_{k=-d+1}^{-1} \upos (P(u) u^k) W_k
\end{align*}
which can be rewritten as
\begin{align*}	
	 (1-zP(u))\sum_{k=c+1}^{\infty} W_k u^k  = 
	    -z(1-q)\sum_{k=-d+1}^{-1} \upos (P(u) u^k) W_k
	 	  -\upos \left( (1-zP(u))\sum_{k=-d+1}^c W_k u^k \right).
\end{align*}
In this expression, it is legitimate to substitute $u$ by the $c$ roots $u_i(z)$ of $1-zP(u)$ such that $u_i(z)\sim 0$ for $z\sim0$.
Thus we get $c$~new equations:
\begin{align*}	
	 0  = \left( 
	    zq\sum_{k=-d+1}^{-1} \upos (P(u) u^k) W_k
		+z \upos (P(u)) W_0 
	 	-\sum_{k=1}^{c} \upos( (1-zP(u))u^k ) W_k \right)_{|u=u_i(z)}.
\end{align*}
\item Extract the negative part $\uneg$. This leads to an equation
in which it is legitimate to substitute $u$ by the $d$ roots $v_i(z)$ of $1-zP(u)$ such that $|v_i(z)|\sim \infty$ for $z\sim0$.
Thus we get $d$~new equations:
\begin{align*}	
	 0  = \left( 
	    zq\sum_{k=1}^{c-1} \uneg (P(u) u^k) W_k
		+z \uneg (P(u)) W_0 
	 	-\sum_{k=-d}^{-1} \uneg( (1-zP(u))u^k ) W_k \right)_{|u=v_i(z)}.
\end{align*}
\item Extract $[u^0]$. It gives one additional equation:
\begin{align*}
[u^0] \left(	(1 - z P(u)) \left( \sum_{k=-d}^c W_k u^k \right)\right) = 1.
\end{align*}

\end{enumerate}
All these equations involve (a subset of) of the yet unknown $c+d+1$ auxiliary functions $W_{-d},\dots,W_c$.
This system fully allows to reconstruct the initial equation: 
$W$ (the linear combination of the $W_k$'s) is also a solution of the initial functional equation, which was a contraction
in the space of formal power series $\C[q,u,1/u][[z]]$, 
therefore this functional equation had a unique solution in this ring, 
which we thus identified as being $W$, the only candidate for~it.
This explains why the system of equations is of full rank: it has unique power series solutions $W_k$,
all of them expressible as a quotient of polynomials having the roots as variable. 
(We say more on this shape in the full version of this article: it requires an excursion (sic) via the Schur function world!) 
So, in all cases, one gets a closed form for the generating function $W$ (and the $W_k$'s) in terms of the roots of the kernel.
\end{proof}

Now, the local analysis of these roots at their branching points, 
as done in~\cite{BaFl02} allows to get the Puiseux behaviour of the functions $W(z,u,1)$ and $W_0(z,u)$ 
near their dominant singularity (for any fixed value of $u$, this dominant singularity is still the same $\rho$ 
as in~\cite{BaFl02}). It allows to rewrite locally 
these functions (for $z\sim \rho$) in the framework of the schemes developed in~\cite{Wallner16b}, on which we comment 
more in the next section.

Let us now illustrate this approach on basketball walks (walks with jumps $-2,-1,0,+1,+2$, see~\cite{KKKK16}).
This leads to the simple linear system:
\begin{equation*}\begin{cases}
W_0=1+z (W_{-1}+W_{-2}+W_0+ W_1 + W_2 ) \\
qzW_{{-1}}+ \left( zu_{{1}}+z \right) W_{{0}}+ \left( z{u_{{1}}}^{2}+zu_{{1}}+z-1 \right) W_{{1}}+ \left( z{u_{{1}}}^{3}+z{u_{{1}}}^{2}+zu_{{1}}+z-u_{{1}} \right) W_{{2}}=0\\
qzW_{{-1}}+ \left( zu_{{2}}+z \right) W_{{0}}+ \left( z{u_{{2}}}^{2}+zu_{{2}}+z-1 \right) W_{{1}}+ \left( z{u_{{2}}}^{3}+z{u_{{2}}}^{2}+zu_{{2}}+z-u_{{2}} \right) W_{{2}}=0\\
qzW_{{1}}{v_{{1}}}^{3}+ \left( v_{{1}}+1 \right) zW_{{0}}{v_{{1}}}^{2}+ \left( (z-1){v_{{1}}}^{2}+zv_{{1}}+z \right) v_1 W_{-1}+ \left( z{v_{1}}^3+(z-1){v_{{1}}}^{2}+zv_{{1}}+z \right) W_{{-2}}=0\\
qzW_{{1}}{v_{{2}}}^{3}+ \left( v_{{2}}+1 \right) zW_{{0}}{v_{{2}}}^{2}+ \left( (z-1){v_{{2}}}^{2}+zv_{{2}}+z \right) v_2 W_{-1}+ \left( z{v_{2}}^3+(z-1){v_{{2}}}^{2}+zv_{{2}}+z \right) W_{{-2}}=0
\end{cases}\end{equation*}
Solving it gives a closed form expression for $W(z,u,q)$ lengthy to write.
If one sets $u=1$, the Ferrari--Bombelli formula can be used and gives 
the following closed form for the generating functions of the walks:
\[W(z,1,q)=
{\frac { z \left( z+1 \right)  \left( q-1 \right) ^{2} \left( 4\,{q}^{2}{z}^{2}-3\,q{z}^{2}-6\,qz+4\,{z}^{2}+q \right)  \sqrt{z-1}}
{2B\sqrt {5\,z-1}}}
+{\frac {\left( {q}^{2}-1 \right) \sqrt [4]{z-1}\sqrt {z}\sqrt {{\it C- D E }}}{\sqrt {2} B \left( 5\,z-1 \right) ^{3/4}}}
-{\frac {A}{2B}}\]
where  $E:=\sqrt{(5z-1)(z-1)}$ and  \\
\resizebox{\textwidth}{!}{\begin{minipage}{\textwidth}
\begin{align*}A:=& 5 \left( {q}^{2}+1 \right)  \left( 4{q}^{2}-3q+4 \right) {z}^{4}+ \left( 2{q}^{4}-49{q}^{3}+24{q}^{2}-49q+2 \right) {z}^{3}- \left( 8{q}^{4}+11{q}^{3}-20{q}^{2}+11q+8 \right) {z}^{2}\\& + \left( 2{q}^{4}+13{q}^{3}+16{q}^{2}+13q+2 \right) z-2q (q+1)^2,
\\
B:=& \left( 31{q}^4+{q}^{3}+61{q}^2+q+31 \right) z^5- \left( 8{q}^{4}+83{q}^{3}+18{q}^2+83q+8 \right) {z}^{4}-2 \left( 7{q}^{4}+3{q}^{3}-15{q}^{2}+3q+7 \right) {z}^{3}\\&+2 \left( 4{q}^{4}+17{q}^{3}+20{q}^{2}+17q+4 \right) {z}^{2}- \left( {q}^{4}+11{q}^{3}+19{q}^{2}+11q+1 \right) z+{q}^{3}+2{q}^{2}+q,
\\
C:=& (z-1) (5z-1) ( z+1)  \Big(  \left( 44{q}^{4}-26{q}^{3}+89{q}^{2}-26q+44 \right) {z}^{5} + \left( 4{q}^{4}-126{q}^{3}+44{q}^{2}-126q+4 \right) {z}^{4}
\\&
+ \left( -20{q}^{4}-14{q}^{3}+58{q}^{2}-14q-20 \right) {z}^{3}+ \left( 4{q}^{4}+46{q}^{3}+24{q}^{2}+46q+4 \right) {z}^{2}+ \left( -8{q}^{3}-27{q}^{2}-8q \right) z+4{q}^{2} \Big),
\\
D:=& 5\left( 34q^4-6q^3+69q^2-6q+34 \right) z^7+2 \left( 22{q}^{4}-243{q}^{3}+67{q}^{2}-243q+22 \right) z^6 
\\&- \left( 132{q}^{4}+196{q}^{3}-81{q}^{2}+196q+132 \right) {z}^{5}-4 \left( 2{q}^{4}-81{q}^{3}-42{q}^{2}-81q+2 \right) {z}^{4}
\\&+ \left( 26{q}^{4}+58{q}^{3}-125{q}^{2}+58q+26 \right) {z}^{3} -2 \left( 2{q}^{4}+31{q}^{3}+29{q}^{2}+31q+2 \right) {z}^{2}+q \left( 8{q}^{2}+35q+8 \right) z-4{q}^{2}.\end{align*}
\end{minipage} }

\smallskip
Note that the Cardan/Ferrari--Bombelli formulas are perhaps nice for the eyes (arguably!),
but they  are not the right way to handle these generating functions from a computer algebra point of view.
It is indeed better to work directly with symmetric functions of the roots $u_i$'s
(and this advantageously also allows to handle the cases of degree~$>4$): 
this is an efficient algorithmic way to use the Newton relations/Vieta's formulas between the roots of the kernel.
Using the expression with the small roots is also the way to follow the Puiseux behaviour
of these symmetric functions,  as we detail in the forthcoming full version.

What is more, the above linear system shows that it is now routine to get in a few minutes thousands of coefficients of our generating functions (e.g.~via Newton iteration, in any computer algebra software).
Here are the first terms of the generating functions of $x$-axis crossings for walks: 

\begin{gather}
\begin{align*}
 W(z,1,q)=&1+5\,z+ \left( 2\,q +23 \right) {z}^{2}+ \left( 2\,{q}^{2}+14\,q+109 \right) {z}^{3}+\\
& \left( 2\,{q}^{3}+16\,{q}^{2}+88\,q+519 \right) {z}^{4}+ \left( 2\,{q}^{4}+18\,{q}^{3}+112\,{q}^{2}+504\,q+2489 \right) {z}^{5}+\\
& \left( 2\,{q}^{5}+20\,{q}^{4}+138\,{q}^{3}+700\,{q}^{2}+2776\,q+11989 \right) {z}^{6}+\\
& \left( 2\,{q}^{6}+22\,{q}^{5}+166\,{q}^{4}+930\,{q}^{3}+4150\,{q}^{2}+14896\,q+57959 \right) {z}^{7}+\\
& \left( 2\,{q}^{7}+24\,{q}^{6}+196\,{q}^{5}+1196\,{q}^{4}+5878\,{q}^{3}+23720\,{q}^{2}+78614\,q+280995 \right) {z}^{8}+\\
& \left( 2\,{q}^{8}+26\,{q}^{7}+228\,{q}^{6}+1500\,{q}^{5}+8004\,{q}^{4}+35518\,{q}^{3}+132264\,{q}^{2}+410046\,q+1365537 \right) {z}^{9}+O \left( {z}^{10} \right) .
\end{align*}
\end{gather}

What is the asymptotic behaviour of these polynomials in $q$? 
This is what we present in the next section.

\newpage
\section{Limit laws}

\begin{theo}[Limit law for the local time at $0$]\label{limitlaw}
For a walk with step set encoded by $P(u)$,
the limit laws for the local time at $0$ depend on
the drift $P'(1)$ of the walk (see Table~\ref{tab:compprobdis}):

\begin{center}
\begin{tabular}{ll}
	Type of the walk & Limit law\\
	\hline\hline	
	Excursion & Negative binomial distribution\\
	Meander with drift $<0$ & Negative binomial distribution\\
	Meander with drift $\geq 0$ & Geometric distribution\\
	Walk with non-zero drift & Geometric distribution\\
	Walk with zero drift & Half-normal distribution ${\mathcal H}(\lambda)$\\
	Bridge & Rayleigh distribution  ${\mathcal R}(\lambda)$
\end{tabular}
\end{center}

The parameters of these distributions are given in the proof and in Table~\ref{tab:compprobdis}:
E.g.~for bridges, the parameter of ${\mathcal R}(\lambda)$ is $\lambda=\sqrt{P''(1)/P(1)}$,
for walks, the parameter of ${\mathcal H}(\lambda)$ is $\lambda=\tau/2 \sqrt{P(1)/P''(1)}$,
where $\tau$ is the unique real positive value such that $P'(\tau)=0$.

What is more, if defined, the limit laws for the the number of $x$-axis crossings are the same as the ones of the local time
(with an appropriate new value for the distribution parameters).
\end{theo}

\begin{proof}[Proof (Sketch)]
Thanks to the Puiseux expansions following from Theorem~\ref{theo:closedform} and \cite{BaFl02}, 
it is possible to derive the limit laws:
we then get a shape on which we can apply the results of Drmota and Soria~\cite{DrSo97} on the Rayleigh and Gaussian distributions,
and of~\cite{Wallner16b} on the half-normal distribution. These distributions are depicted in Table~\ref{tab:compprobdis}. 
Details are omitted in this extended abstract.

In the cases of excursions and meanders the local time is equal to the number of returns to zero. 
The results for excursions were derived by Banderier and Flajolet in \cite[Theorem~5]{BaFl02}.
What is more, all ingredients for the case of meanders are also given in this paper, 
and the result follows the same lines. 
However, due to the drift dependent number of meanders (compare \cite[Theorem~4]{BaFl02}) 
three regimes need to be considered, leading to two different limit laws: negative binomial and geometric.

\def\arraystretch{1.9}
\begin{table}[hbt]
	\begin{center}
	\begin{tabular}{|c||c|c|c|c|}
		\hline
		&		\underline{\bf Geometric}
		&		\underline{\bf Negative binomial}
		&		\underline{\bf Half-normal}
		&		\underline{\bf Rayleigh}
		\\
		&		$\operatorname{Geom}(p)$
		&		$\operatorname{NB}(m,p)$
		&		$\Hc(\lambda)$
		&		$\Rc(\lambda)$
		\\		\hline 
		&		 {\includegraphics[width=0.20\textwidth]{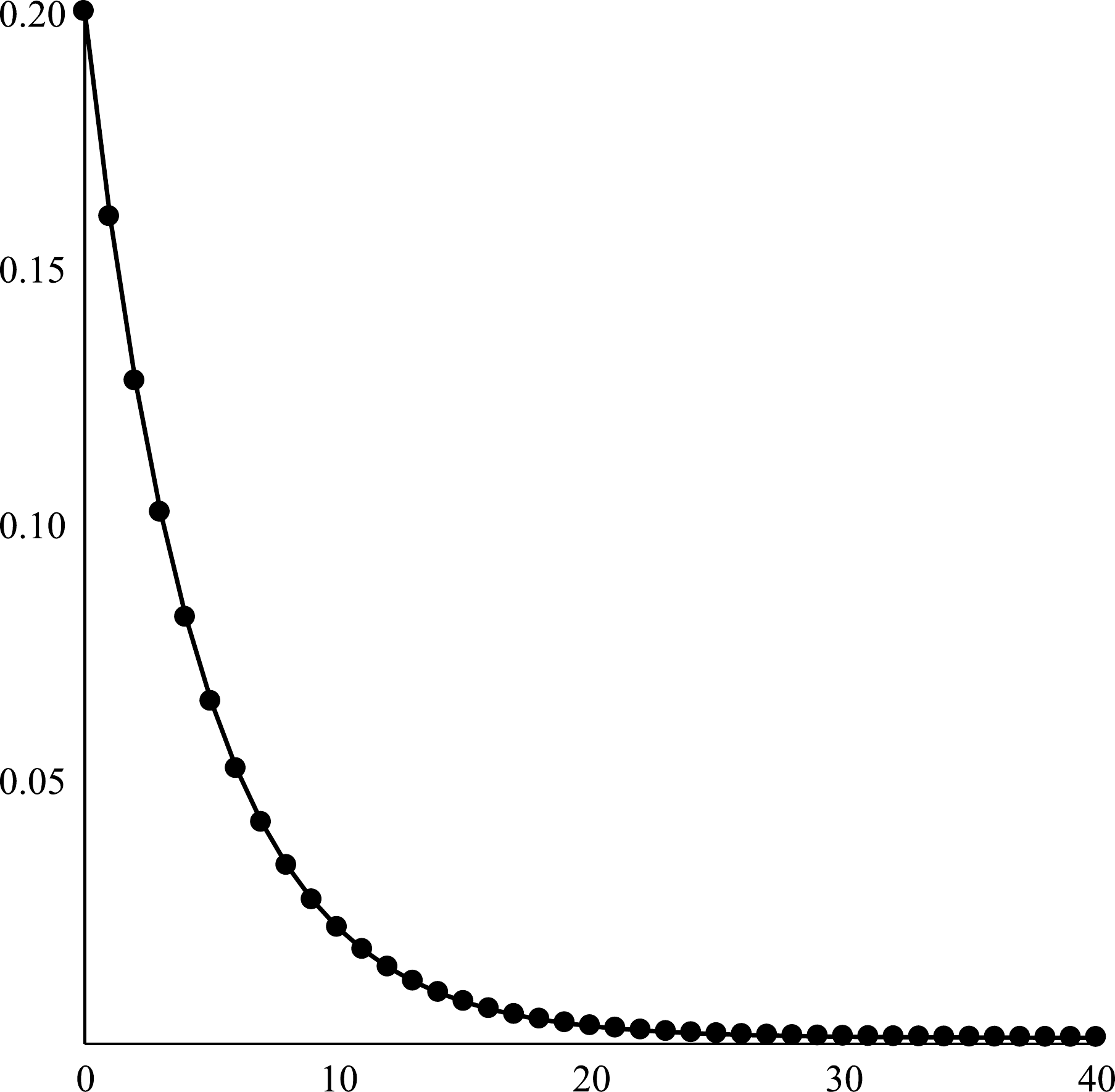}} 
		&		 {\includegraphics[width=0.20\textwidth]{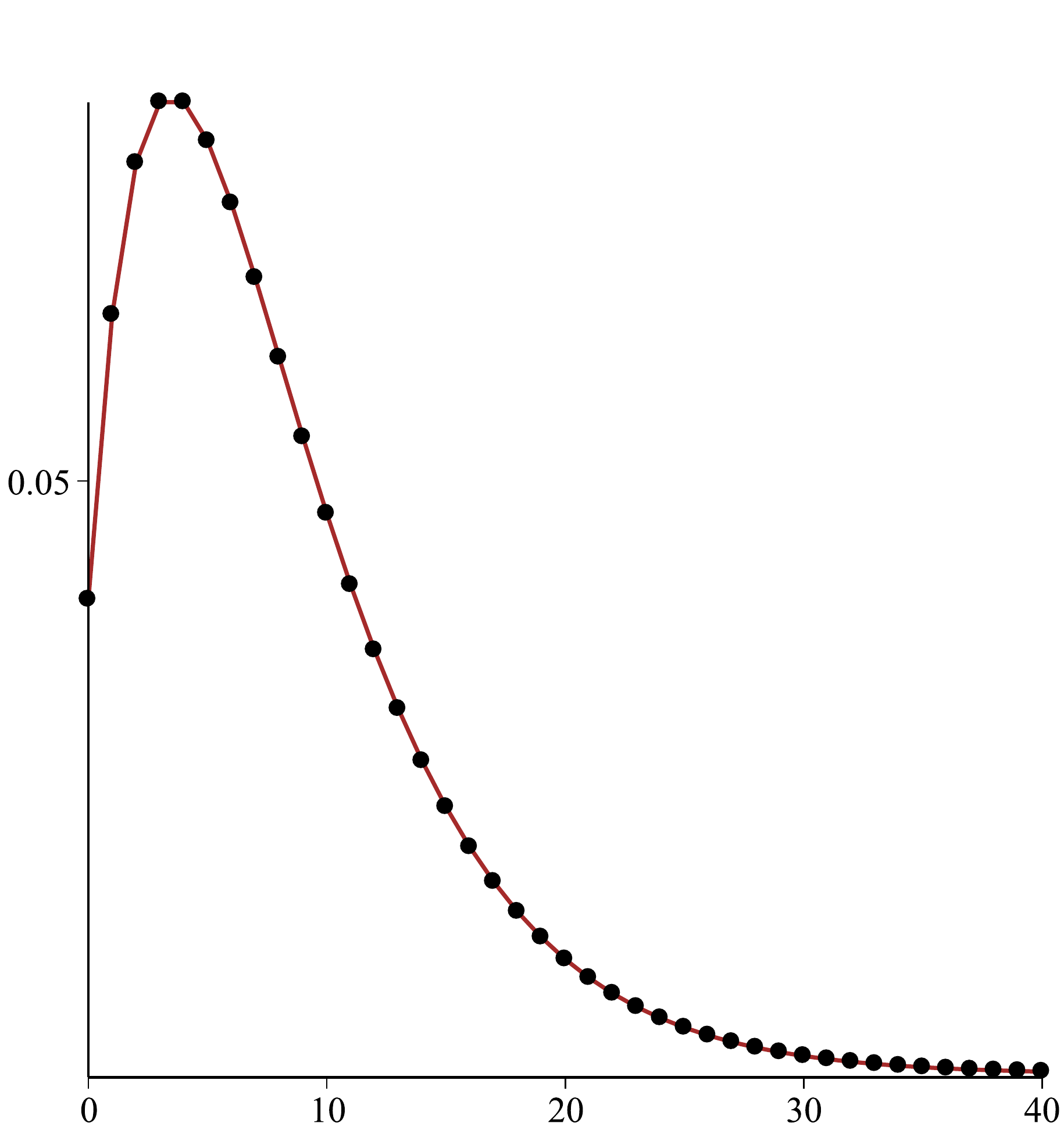}} 
		&		 {\includegraphics[width=0.20\textwidth]{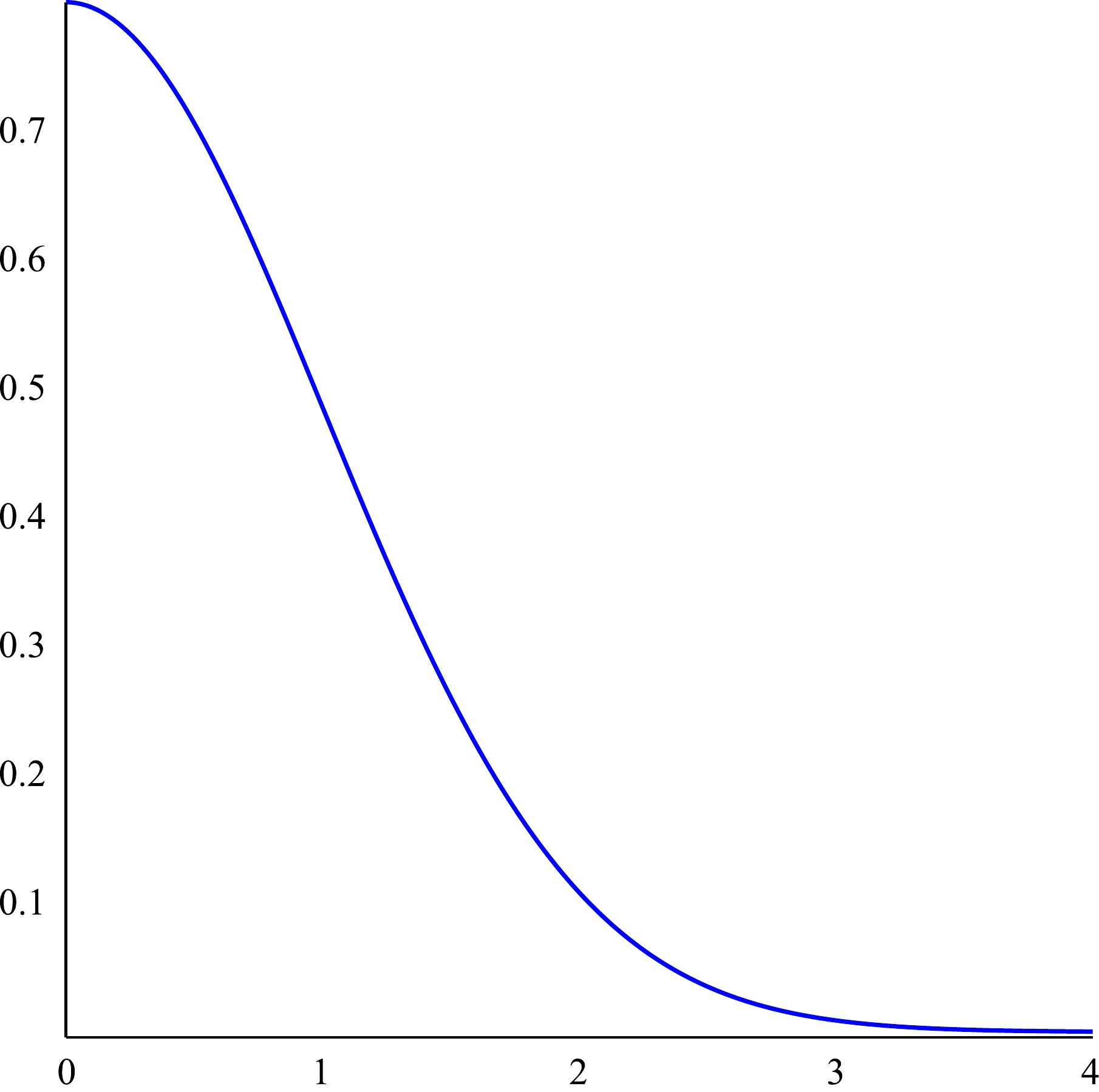}} 
		&		 {\includegraphics[width=0.20\textwidth]{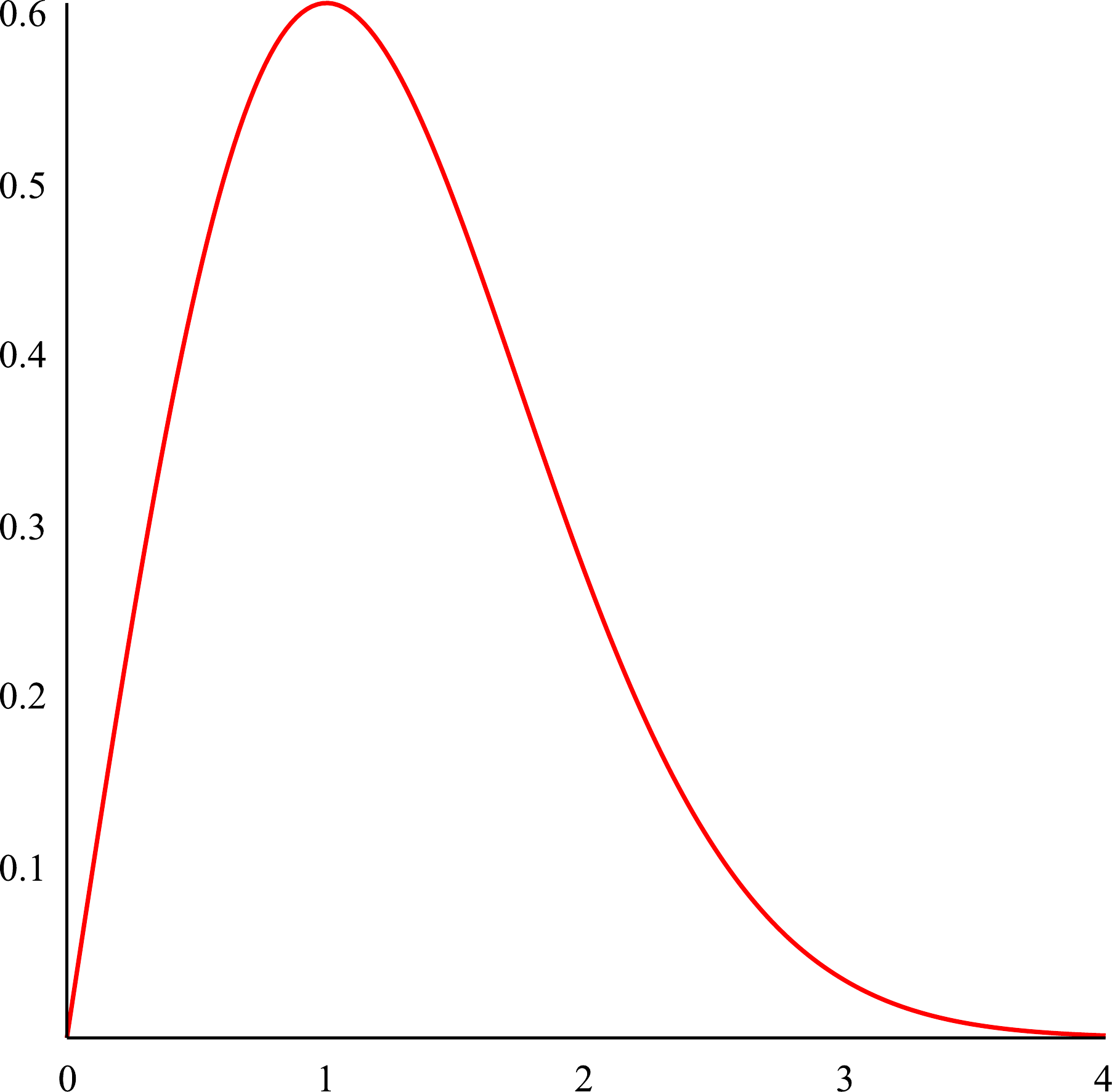}} 
		\\		\hline		Support
		&		$k \in \{0,1,\ldots\}$
		&		$k \in \{0,1,\ldots\}$
		&		$x \in \R_{\geq 0}$
		&		$x \in \R_{\geq 0}$
		\\		\hline		PDF
		&		$(1-p)^k p$
		&		$\binom{m+k-1}{k}(1-p)^k p^m$
		&		$\sqrt{\frac{2}{\pi \lambda^2}} \exp\left(-\frac{x^2}{2\lambda^2}\right)$
		&		$\frac{x}{\lambda^2} \exp\left(-\frac{x^2}{2\lambda^2}\right)$
		\\		\hline		Mean
		&		$\frac{1-p}{p}$
		&		$\frac{m(1-p)}{p}$
		&		$\lambda\sqrt{\frac{2}{\pi}}$
		&		$\lambda\sqrt{\frac{\pi}{2}}$
		\\		\hline		Variance
		&		$\frac{1-p}{p^2}$
		&		$\frac{m(1-p)}{p^2}$
		&		$\lambda^2 \left(1 - \frac{2}{\pi}\right)$
		&		$\lambda^2 \left(2 - \frac{\pi}{2} \right)$
		\\
		\hline
	\end{tabular}
	\end{center}
	\caption{A comparison of the geometric, negative binomial, half-normal, and Rayleigh distribution. They 
are the four distributions occurring for the local time of generalized Dyck paths. 
Feller~\cite[Chapter 3]{Feller68} did the analysis for the simplest case of walks with jumps $-1,+1$, 
our article shows how to tackle the more difficult situation involving any set of jumps of amplitude $\geq 1$.}
	\label{tab:compprobdis}
\end{table}

\newpage

For the cases of walks and bridges the local time is equal to the number of returns to zero and the number of $x$-axis crossings. 
The results for the number of returns to zero of walks were derived by Wallner in~\cite[Theorem~4.2]{Wallner16b} leading to a geometric or a half-normal distribution, depending whether the drift $P'(1)$ is non-zero or zero, respectively. 
As above, the result for the case of returns to zeros in bridges, follows also the same lines as the previous one and we omit the details in this extended abstract. In this case one uses the limit law of Drmota and Soria~\cite[Theorem~1]{DrSo97} to prove the existence of a Rayleigh distribution.

It remains to consider the laws of $x$-axis crossings in the cases of bridges and walks.
The proof of Theorem~\ref{theo:closedform} gives access to closed forms of $W_0(z,q)$ and $W(z,1,q)$, which are the generating functions of bridges and walks where crossings of the $x$-axis are marked by $q$.
In order to treat both at the same time we abbreviate them until the end of this proof by $F(z,q)$.
We want to apply either \cite[Theorem~1]{DrSo97} or \cite[Theorem~2.1]{Wallner16b}. 
Note that the technical conditions of these theorems are satisfied due to the closed forms in terms of small and large branches, 
and the fact, that due to the Weierstrass Preparation Theorem, 
the branches $u_1(z)$ and $v_1(z)$ (which are the real positive branches for $z>0$ in the vicinity of $0$, see Figure~\ref{Figu1v1})
 satisfy the necessary conditions (compare the derivations in \cite{Wallner16b}). 
In particular, as proven in~\cite{BaFl02}, they satisfy a square root behaviour with the following expansion at $z=\rho$:
\begin{align*}
	u_1(z) &= \tau - \sqrt{2\frac{P(\tau)}{P''(\tau)}} \sqrt{1-z/\rho} + \ldots,\\
	v_1(z) &= \tau + \sqrt{2\frac{P(\tau)}{P''(\tau)}} \sqrt{1-z/\rho} + \ldots.
\end{align*}
Due to the methods derived in~\cite{BaWa16}, we may assume without loss of generality, that our model is aperiodic. 
For such aperiodic models (like e.g.~Motzkin and basketball walks) there exists a unique singularity $\rho>0$ of $F(z,1)$. 

Now, the key fact is that our generating functions $F$ have locally the following behaviour
\begin{align*}
	\frac{1}{F(z,q)} &= g(z,q) + h(z,q)\sqrt{1-z/\rho},
\end{align*} 
for $|q-1|<\varepsilon$ and $|z-\rho|<\varepsilon$ with $\arg(z-\rho) \neq 0$ where $\varepsilon >0$ is some fixed real number, and $g(z,q)$ and $h(z,q)$ are analytic functions. 
Additionally, one has here that $g(\rho,1) = 0$.
Finally, in the case of bridges we get $g_q(\rho,1) <0$ and $h(\rho,1) \neq 0$ yielding a Rayleigh law. 
Whereas, in the case of walks with zero drift we get $g_q(\rho,1)=g_{qq}(\rho,1)=0$ and $h(\rho,1)=0$, but $g_z(\rho,1) \neq 0$ and $h_q(\rho,1) \neq 1$ giving a half-normal law. 
The respective parameters depend on the chosen step set. 
\end{proof}

\begin{figure}[ht!]
\begin{center}
\includegraphics[width=.5\textwidth]{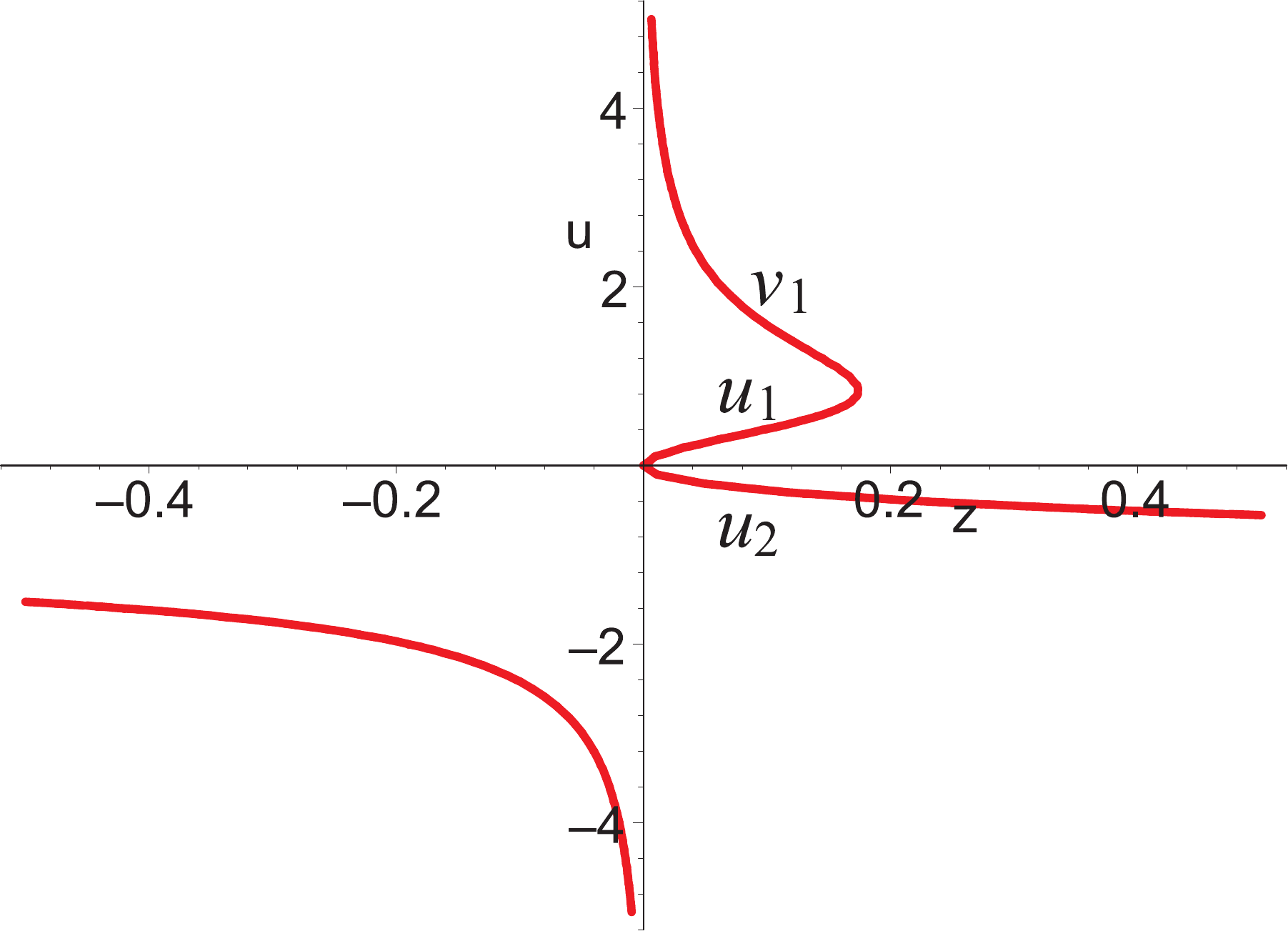}
\end{center}
\caption{The root $u_1$ and $v_1$ of the kernel $1-zP(u)$ are crossing at $z=\rho$ with a square root behaviour.
This implies the Puiseux expansions with a square root used in the proof above.}
\label{Figu1v1}
\end{figure}

\newpage

\def\w{84mm}

\begin{figure}[ht!]
\begin{center}
\includegraphics[width=\w]{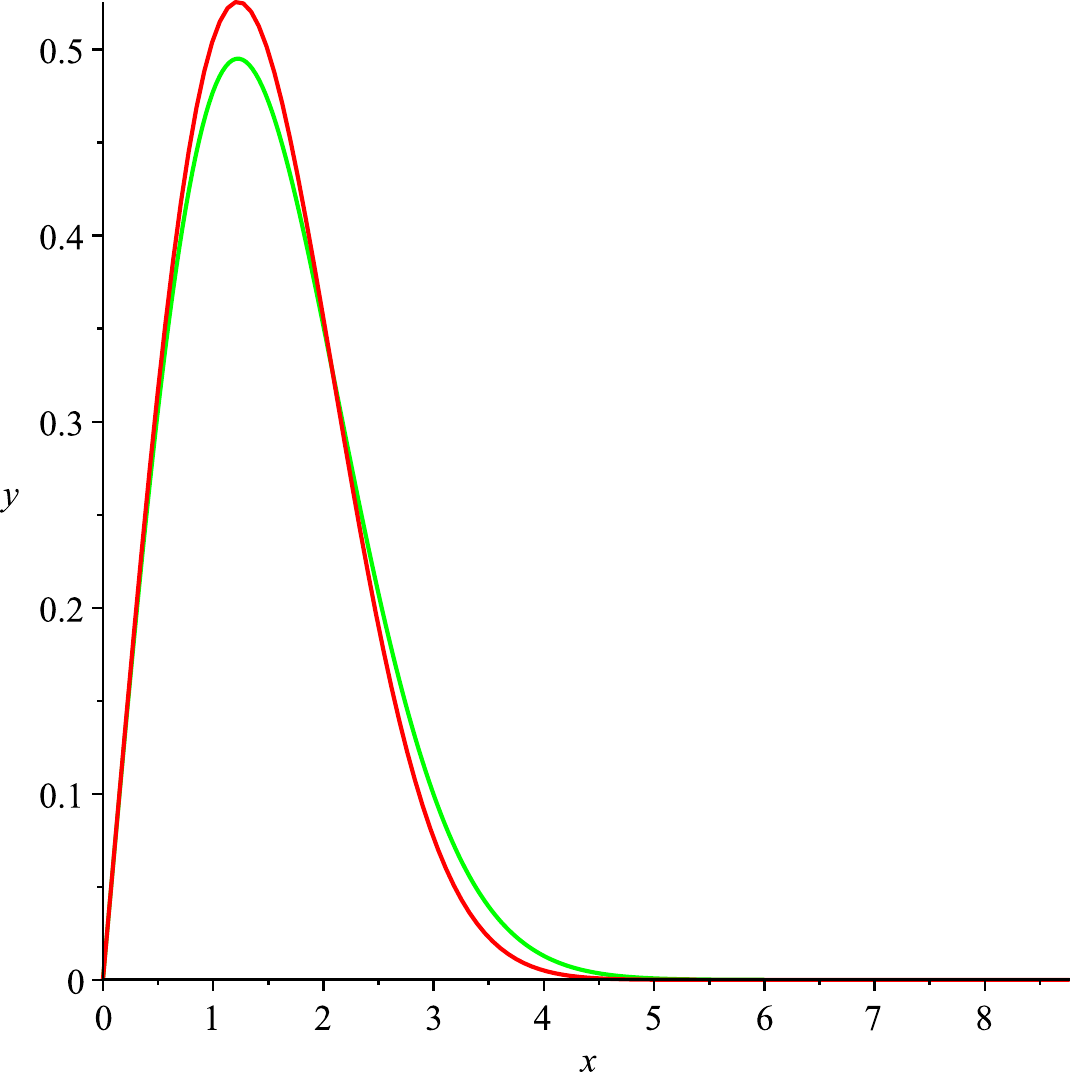}
\includegraphics[width=\w]{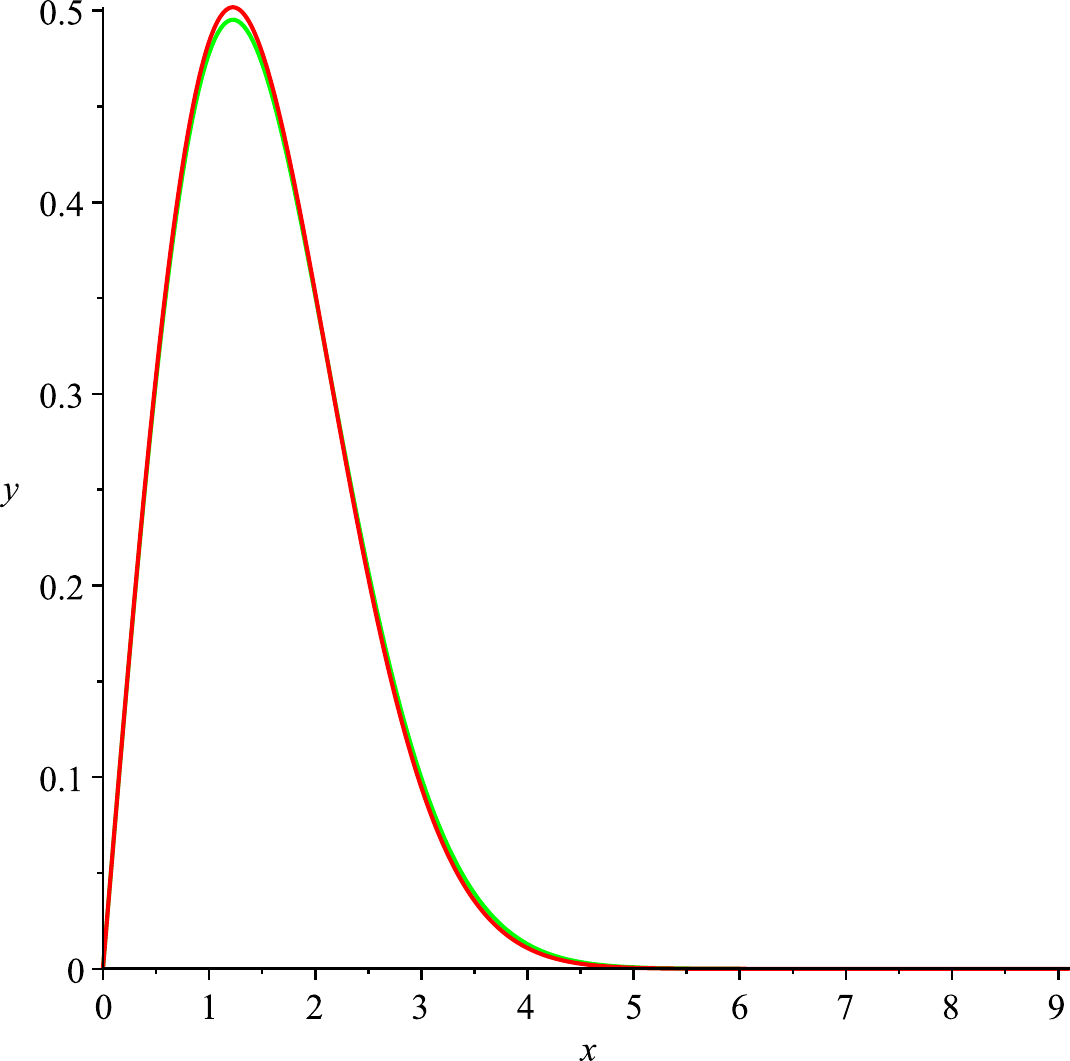}
\end{center}
	\caption{Distribution of $X_n$, the local time at $0$ for bridges of length $n$. 
(Left: $n=200$. Right: $n=4000$. Step set ${\mathcal S}=[-2,-1,0,1,2]$).
One quickly observes a nearly perfect match of the distribution (in red) of $X_n$
with the Rayleigh limit law (in green).
This fast convergence is well explained by the error terms we get
via analytic combinatorics.}
\end{figure}

\begin{figure}[hb!]
\begin{center}
\includegraphics[width=\w]{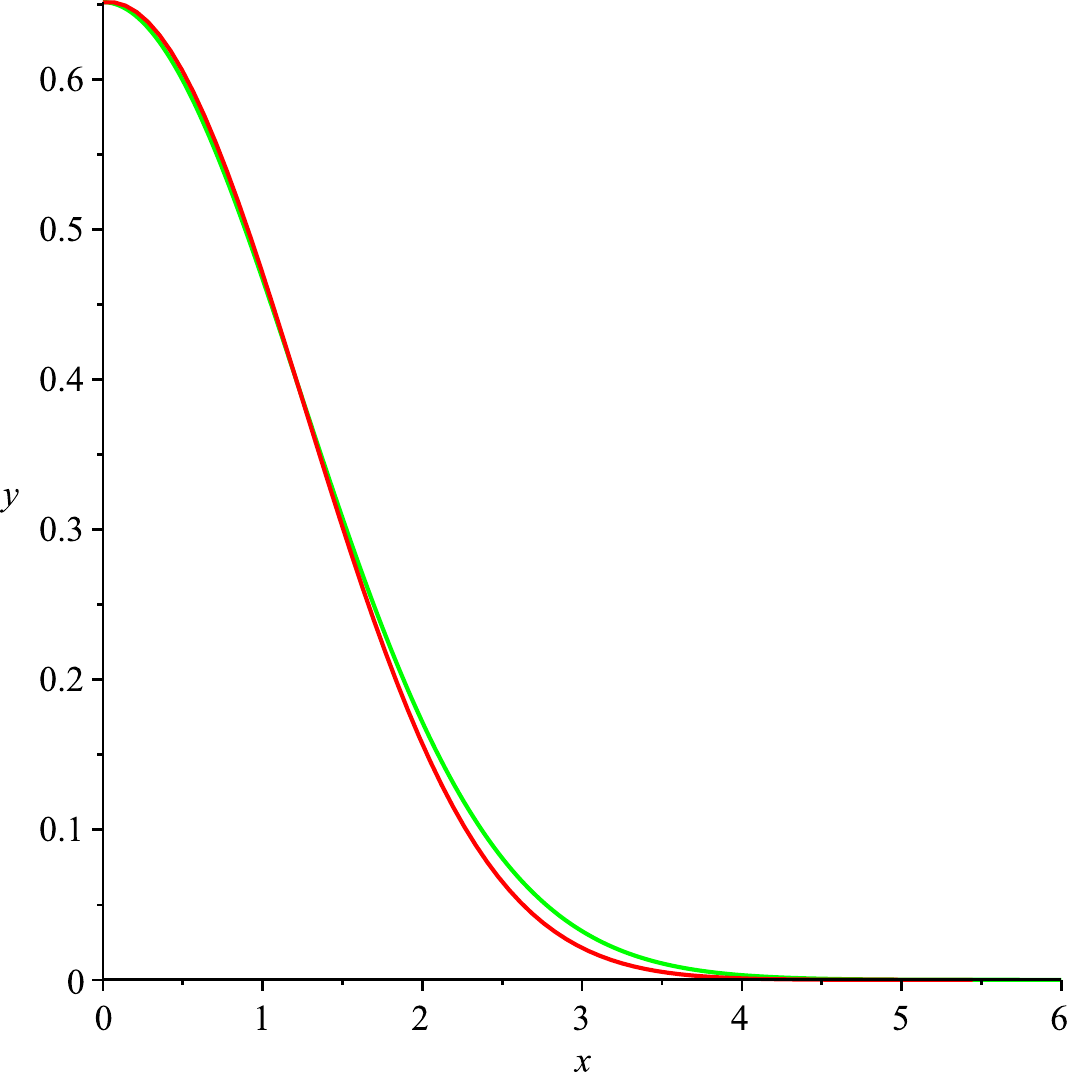}
\includegraphics[width=\w]{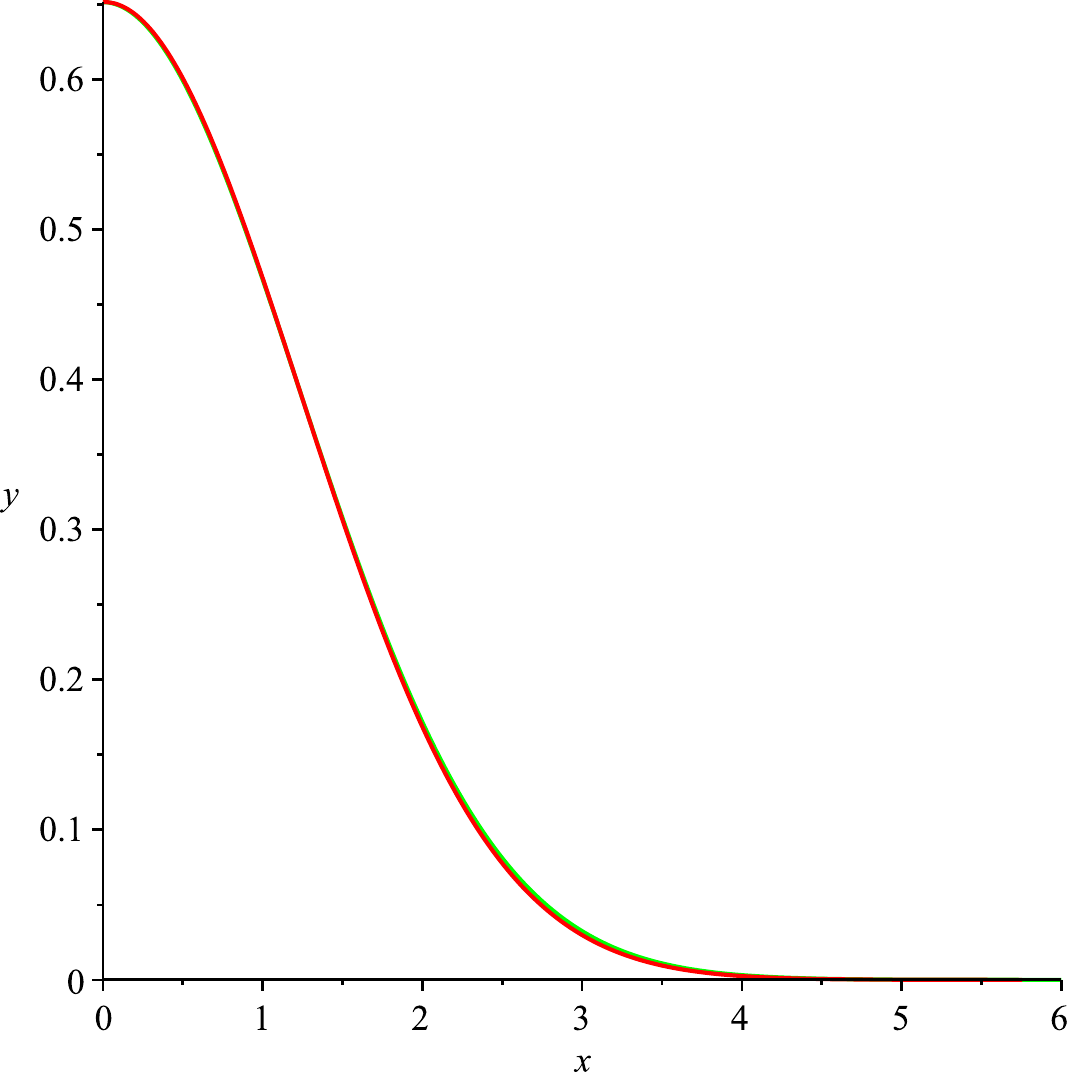}
\end{center}
	\caption{Distribution of $X_n$, the local time at $0$ for walks of length $n$. 
(Left: $n=200$. Right: $n=4000$. Step set ${\mathcal S}=[-2,-1,0,1,2]$).
One quickly observes a nearly perfect match of the distribution (in red) of $X_n$
with the half-normal limit law (in green). Here again,
this fast convergence is well explained by the error terms we get
via analytic combinatorics.}
\end{figure}
 
\pagebreak

\section{Conclusion}
\label{sec:conclusion}
In this article we showed how to derive the generating function for the local time at $0$ of directed lattice paths. 
It completes the work of Banderier and Flajolet~\cite{BaFl02}, 
who just handled the case of returns to zero of excursions.
It is also extending the work of Feller~\cite{Feller68} and later Wallner~\cite{Wallner16b},
who did the Dyck and Motzkin cases (for meanders/walks).

In order to solve the generating functions in the more general case, 
we used a mixture of the Omega operator and the kernel method. This leads to expressions
from which we showed how to derive the limit law of the local time, for several models of constrained lattice paths.
In the full version of this article, we give more closed form formulas, and we show that other parameters 
(like the number of ``changes of signs'', or the number of jumps from positive to negative altitude)
can also be analysed using our approach, and that they satisfy similar limit laws. 

These parameters are very natural for discrete random walks,
it is interesting that it is not possible to analyse them via a Brownian motion approach:
indeed a Brownian motion can be seen as the limit after a rescaling of the amplitude  
by~$\sqrt{n}$ and the length by~$n$ (see~\cite{Marchal03}).
This rescaling implies that (discrete) jumps crossing the abscissa would be of amplitude $0$, and are therefore completely erased.
It is therefore nice that analytic combinatorics can get the asymptotics of this ``discrete local time'', 
and the corresponding universal limit laws, while there are ``invisible'' via a Brownian motion approach.
In a forthcoming article, we tackle further analysis of the height of discrete lattice paths;
this allows to get the connection with the Brownian local time.

\smallskip
\textbf{Acknowledgments:}
\label{sec:ack}
This work was started via collaboration funded by the SFB project F50 ``Algorithmic and Enumerative Combinatorics'' 
and the Franco-Austrian PHC ``Amadeus'', and ended during 
the postdoctoral position of Michael Wallner at the University of Paris Nord, in September-December 2017, thanks to a MathStic funding. 
Michael Wallner is currently supported by the Erwin Schr{\"o}dinger Fellowship of the Austrian Science Fund (FWF):~J~4162-N35.

\bibliographystyle{alpha}

\newcommand{\etalchar}[1]{$^{#1}$}

\end{document}